\documentclass[11pt, twoside]{article}

\usepackage[english]{babel}

\usepackage{amssymb}
\usepackage{mathrsfs}
\usepackage{amsmath}
\usepackage{amsthm}
\usepackage{amsfonts}
\usepackage{latexsym}
\usepackage{indentfirst}
\usepackage{color}
\usepackage{txfonts}
\usepackage{enumerate}

\usepackage[colorlinks=true,
linkcolor=blue,
citecolor=red,
urlcolor=magenta,
]{hyperref}

\usepackage{txfonts}
\usepackage{anysize}

\allowdisplaybreaks

\newtheorem{theorem}{Theorem}[section]
\newtheorem{lemma}[theorem]{Lemma}

\newtheorem{proposition}[theorem]{Proposition}

\theoremstyle{definition}

\newcounter{assum}

\renewcommand{\appendix}{\par
\setcounter{section}{0}%
\setcounter{subsection}{0}%
\setcounter{subsubsection}{0}%
\gdef\thesection{\@Alph\c@section}%
\gdef\thesubsection{\@Alph\c@section.\@arabic\c@subsection}%
\gdef\theHsection{\@Alph\c@section.}%
\gdef\theHsubsection{\@Alph\c@section.\@arabic\c@subsection}%
\csname appendixmore\endcsname
}

\numberwithin{equation}{section}

\begin{document}
\title{\bf\Large
A Counterexample to the Necessity of the Vanishing Carleson 
Condition for VMO Poisson Kernels\footnotetext{\hspace{-0.35cm} 2020 {\it
Mathematics Subject Classification}.
Primary 35J25; Secondary 35B20, 31B05, 42B37.
\endgraf {\it Key words and phrases.}
elliptic operator, Dirichlet problem, Poisson kernel, elliptic measure,
vanishing Carleson condition,
VMO space.
\endgraf This project is partially supported by the National
Natural Science Foundation of China (Grant Nos. 12431006, 12371093, and 12501118),
the Beijing Natural Science Foundation (Grant No. 1262011),
the Fundamental Research Funds for the Central Universities
(Grant No. 2253200028), and Longyuan Young Talents of Gansu Province.
}}
\author{Xiaosheng Lin, Dachun Yang\footnote{Corresponding
author, E-mail: \texttt{dcyang@bnu.edu.cn}/{\color{red}\today}/Final version.},
\ \ Sibei Yang, Wen Yuan
and Yangyang Zhang}
\date{}
\maketitle

\vspace{-0.7cm}

\begin{center}
\begin{minipage}{13cm}
{\small {\bf Abstract.}\quad
In [Problem 3.2.23, CBMS Regional Conference Series in Mathematics 83, 1994],
Kenig asked whether the vanishing Carleson condition is the
necessary and sufficient for the logarithm
of the Poisson kernel of a perturbation of the Laplacian on the unit ball in
$\mathbb{R}^n$ belonging to the space VMO. The sufficiency was proved by
Escauriaza [Israel J. Math. 1996] and extended by Milakis, Pipher, and Toro
[Contemp. Math. 2014] to more general setting. In this article,
using the technique of bi-Lipschitz mappings, we construct a counterexample
to show that the vanishing Carleson condition is not necessary and hence give a
negative answer to the aforementioned problem.}
\end{minipage}
\end{center}

\vspace{0.1cm}

\tableofcontents
\section{Introduction}

In the study of second-order elliptic partial differential equations 
in divergence form, a key topic is
the boundary regularity of solutions and the associated elliptic measure (see, for example,
\cite{ahmmt20,amt17,btz23a,cfk81,d77,fkp91,Kenig1994,t2010}). Let $n\ge2$,
$\Omega\subset\mathbb{R}^n$ be a domain, and
$L:=\operatorname{div}(A\nabla\cdot)$ be a uniformly elliptic
operator on $\Omega$ with real, bounded, measurable coefficients.
For each $x\in\Omega$, the \emph{elliptic measure $\omega_L^x$}
(with pole at $x$) is a family of probability measures on
the boundary $\partial\Omega$ such that,
for any $f\in C(\partial\Omega)$ (the set of all continuous functions
on $\partial\Omega$), the solution $u$ to the Dirichlet problem
\begin{align}\label{eq1.1}
\begin{cases}
Lu=0&\ \text{in}\ \ \Omega,\\
u=f&\ \text{on}\ \ \partial\Omega
\end{cases}
\end{align}
is given by
$$u(x)=\int_{\partial\Omega} f\,d\omega_L^x\quad \mathrm{for\ all}\quad x\in\Omega
$$
(see, for example, \cite{d77,fkp91,Kenig1994}).

A classical result due to Dahlberg \cite{d77} states that, for the Laplacian
on a bounded Lipschitz domain, the elliptic measure is absolutely continuous with respect
to the surface measure $\sigma$ and its Radon--Nikodym derivative, also called
the Poisson kernel, satisfies a reverse H\"older inequality.
Later, Jerison and Kenig \cite{jk82} proved that, in the case of $C^1$ domains, the
logarithm of the Poisson kernel belongs to the space of vanishing mean
oscillation (for short, the space VMO; see, for example, \cite{s75}). For domains with lower regularity
such as chord-arc domains, the analogous conclusion was established by Kenig and
Toro \cite{kt97} for the Laplacian, provided the domain is a \emph{vanishing chord-arc domain}
(that is, its boundary is asymptotically flat and the surface
measure is asymptotically the Lebesgue measure).

In the variable coefficient setting, the situation is more delicate. Caffarelli,
Fabes, and Kenig \cite{cfk81} constructed examples of operators with continuous coefficients
for which the elliptic measure is singular with respect to the surface measure, and hence
the $L^p$ Dirichlet problem for the problem \eqref{eq1.1} cannot be solved for any $p\in(1,\infty)$.
These examples arise from quasiconformal mappings and show that absolute continuity
of elliptic measures with respect to surface measures is not automatic.

Nevertheless, a robust perturbation theory has been developed by Dahlberg \cite{d86} via
introducing the \emph{Carleson measure condition} on the difference of coefficient
matrices. Precisely, let $n\ge2$ and $\Omega\subset\mathbb{R}^n$ be a bounded Lipschitz domain.
Denote the surface measure on $\partial\Omega$ by $\sigma.$ For a uniformly elliptic
operator $L$ in divergence form on $\Omega$, denote by $k_L$ its Poisson kernel
on $\Omega$, namely $k_L=d\omega_L/d\sigma$. For $j\in\{0,1\}$, let
\begin{align*}
L_j:=\operatorname{div}(A_j\nabla\cdot)
\end{align*}
be two real, symmetric, uniformly elliptic operators in divergence form on $\Omega$.
For any $x\in \Omega$, define
\begin{align*}
a(x):=\sup\{|A_1(z)-A_0(z)|:z\in B(x,\delta(x)/2)\},
\end{align*}
where $\delta(x):=\operatorname{dist\,}(x,\partial \Omega):=\inf\{|x-y|:y\in\partial\Omega\}$. For any $Q\in\partial\Omega$
and $r\in(0,\mathrm{diam\,}(\Omega))$, let
\begin{align*}
h(r,Q):=\left\{\frac{1}{\sigma(\Delta(Q,r))}\int_{T(Q,r)}\frac{[a(x)]^2}
{\delta(x)}\,dx\right\}^{\frac12},
\end{align*}
where
$$\mathrm{diam\,}(\Omega):=\sup\{|x-y|: x,y\in\Omega\},$$
$\Delta(Q,r):=B(Q,r)
\cap\partial\Omega$ with $B(Q,r):=\{x\in\mathbb{R}^n: |x-Q|<r\}$, and $T(Q,r):=B(Q,r)\cap\Omega$.
Dahlberg \cite{d86} proved that, if the Carleson norm
$$\sup_{r\in(0,\mathrm{diam\,}(\Omega))}\sup_{Q\in\partial\Omega}h(r,Q)$$
is sufficiently small, then the reverse H\"older inequality properties of
elliptic measures $\{\omega_{L_0},\omega_{L_1}\}$ are preserved.
That is, if $k_{L_0}$ satisfies the reverse H\"older inequality of order $q$ with some $q\in(1,\infty)$,
then $k_{L_1}$ also satisfies the reverse H\"older inequality of order $q$.
Furthermore, Fefferman, Kenig, and Pipher \cite{fkp91} showed that, under the Carleson measure condition
$$\sup_{r\in(0,\mathrm{diam\,}(\Omega))}\sup_{Q\in\partial\Omega}h(r,Q)<\infty,$$
the $A_\infty$ weight property of the elliptic measure is also preserved.

Moreover, if $g$ is integrable on a surface ball
$\Delta\subset\partial\Omega$ with $\sigma(\Delta)>0$,
define its \emph{average} on $\Delta$ by setting
$$
g_{\Delta}:=\frac{1}{\sigma(\Delta)}\int_{\Delta}g\,d\sigma.
$$
Then the \emph{space $\operatorname{BMO}(\partial\Omega,\sigma)$} is defined 
as the set of all integrable functions $g$ on $\partial\Omega$ such that 
\begin{align*}
\|g\|_{\operatorname{BMO}(\partial\Omega,\sigma)}:=\sup_{Q\in\partial\Omega}
\sup_{r\in(0,\mathrm{diam\,}(\Omega))}
\frac{1}{\sigma(\Delta(Q,r))}
\int_{\Delta(Q,r)}|g-g_{\Delta(Q,r)}|\,d\sigma<\infty
\end{align*}
(see, for example, \cite{jn61}), and the \emph{space 
$\operatorname{VMO}(\partial\Omega,\sigma)$} is defined as the set
of all $g\in\operatorname{BMO}(\partial\Omega,\sigma)$ such that 
\begin{align*}
\lim_{\rho\to 0^+}
\sup_{Q\in\partial\Omega}\sup_{r\in(0,\rho]}
\frac{1}{\sigma(\Delta(Q,r))}
\int_{\Delta(Q,r)}|g-g_{\Delta(Q,r)}|\,d\sigma=0
\end{align*}
(see, for example, \cite{s75}). Here, and thereafter,
the notation $\rho\to 0^+$ means that $\rho\in (0,\infty)$
and $\rho\to 0$. For the finer VMO regularity of $\log k_L$,
Escauriaza \cite{Escauriaza1996} proved that, under the
vanishing Carleson perturbation condition, i.e.,
$$\lim_{r\to 0^+}\sup_{Q\in\partial\Omega}h(r,Q)=0,$$
if $\log k_{L_0}\in\mathrm{VMO}(\partial\Omega,\sigma)$, then $\log k_{L_1}\in\mathrm{VMO}(\partial\Omega,\sigma)$.
This result was further extended to chord-arc domains by Milakis,
Pipher, and Toro \cite{mpt14}; see also \cite{amt17,btz23,kt03,kt99,tt24,t2010}
for further results.

From now on, let $B:=B({\bf 0},1)$ be the unit ball 
in $\mathbb{R}^n$, where ${\bf 0}$ denotes the \emph{origin} of
$\mathbb{R}^n$. In his monograph \cite[Problem 3.2.23]{Kenig1994},
Kenig posed the following two-part problem.

\paragraph{Forward implication.}
Assume that
\begin{align*}
\varlimsup_{r\to 0^+}\sup_{Q\in\partial B}h(r,Q)=0,
\end{align*}
and that both $\omega_{L_0}^{\mathbf{0}}$ and $\omega_{L_1}^{\mathbf{0}}$ are absolutely continuous with respect to the
surface measure $\sigma$ on $\partial B$ (so that their Poisson kernels $k_{L_0},k_{L_1}$ exist).
If $\log k_{L_0}\in\mathrm{VMO}(\partial B,\sigma)$, then $\log k_{L_1}\in\mathrm{VMO}(\partial B,\sigma)$.

\paragraph{Suggested converse.}
Assume that $L_0:=\Delta$ (the Laplacian). Then, by the mean value property,
one finds that $\omega_\Delta^{\mathbf{0}}=\frac {\sigma}{\sigma(\partial B)}$,
and hence $k_\Delta$ is constant and
$\log k_\Delta\in\mathrm{VMO}(\partial B, \sigma)$. Assume further that
\begin{align*}
\sup_{r\in(0,1)}\sup_{Q\in\partial B}h(r,Q)<\infty
\end{align*}
and $\log k_{L_1}\in\mathrm{VMO}(\partial B, \sigma)$. Then 
\begin{align*}
\varlimsup_{r\to 0^+}\sup_{Q\in\partial B}h(r,Q)=0.
\end{align*}

As pointed out above, the forward implication was proved by Escauriaza \cite{Escauriaza1996}
for Lipschitz domains and later extended to chord-arc domains by Milakis, Pipher and Toro
\cite{mpt14}. Therefore, the first part of \cite[Problem 3.2.23]{Kenig1994} is settled.

The suggested converse, however, remains open. The goal of this article is to answer this question.
This question asks whether the vanishing Carleson condition is
necessary for the VMO property of the logarithm of the Poisson kernel to be preserved when
the reference operator is the Laplacian. If true, it would give 
an optimal characterization: the vanishing
Carleson condition is equivalent to the VMO regularity of
the logarithm of the Poisson kernel for perturbations of the Laplacian.

In this article, by constructing a counterexample,
we show that the suggested converse by Kenig in
\cite[Problem 3.2.23]{Kenig1994} is \emph{false}. Precisely, we
have the following theorem.

\begin{theorem}\label{thm:main-counterexample}
Let $B:=B({\bf 0},1)$ be the unit ball in $\mathbb{R}^n$, and let $L_0:=\Delta$.
Then there exists a real, symmetric, uniformly elliptic $n\times n$
matrix $A$ on $B$ such that, for the associated operator  $L_1:=\operatorname{div}(A\nabla \cdot),$
the following statements hold.
\begin{enumerate}[{\rm(i)}]
\item $\omega_{L_1}^\mathbf{0}=\frac{\sigma}{\sigma(\partial B)}.$
Consequently, $k_{L_1}\equiv \frac{1}{\sigma(\partial B)}$ and
$\log k_{L_1}\in\operatorname{VMO}(\partial B,\sigma)$.
\item
\begin{align}\label{zhugou}
\sup_{r\in(0,1)}\sup_{Q\in\partial B}h(r,Q)<\infty.
\end{align}
\item
\begin{align}\label{xiajie}
\varlimsup_{r\to0^+}\sup_{Q\in\partial B}h(r,Q)>0.
\end{align}
\end{enumerate}
\end{theorem}

The construction of the matrix $A$ in Theorem \ref{thm:main-counterexample} is 
quite non-trivial. Instead of carrying it out by prescribing its coefficients 
directly, we first construct a carefully designed bi-Lipschitz homeomorphism $\Phi:\mathbb{R}^n\to\mathbb{R}^n$,
and then define $A$ as the coefficient matrix obtained by writing the Laplacian 
in the new coordinates induced by $\Phi$; see Proposition \ref{lem:pullback}. 
This pullback point of view is the key mechanism of the counterexample. 
The map $\Phi$ fixes the origin ${\bf0}$, fixes the boundary of the ball $B$ pointwise, 
and is a localized perturbation of the identity map; see Proposition \ref{lem:Phi}. 
Consequently, although the coefficients of the resulting operator are changed inside $B$, 
the boundary and the pole are kept fixed. This forces the elliptic measure 
of the pullback operator $L_1$ at the origin to coincide exactly with that of 
the Laplacian, namely the normalized surface measure. 
Hence, $k_{L_1}$ is constant and $\log k_{L_1}\in\mathrm{VMO}(\partial B,\sigma)$.

The remaining part of the construction is designed to make the Carleson 
perturbation non-vanishing. The perturbation of $\Phi$ is localized 
near a sequence $\{E_k\}_{k\in\mathbb{N}}$ of pairwise disjoint small 
balls accumulating at the north pole $Q_0:=e_n$ of $B$; see \eqref{e2.4x}. 
On each core ball $E_k$, the map $\Phi$ acts as a fixed anisotropic 
stretch in the first coordinate direction. Therefore, the corresponding 
pullback coefficient matrix $A$ differs from the identity by an 
amount comparable to $\varepsilon$ on every $E_k$. Since these balls 
occur at infinitely many scales approaching $Q_0$, the Carleson 
quantity $h(r,Q_0)$ has a positive lower bound along the sequence 
of radii comparable to $\rho_k$. Thus,
$\varlimsup\limits_{r\to0^+}h(r,Q_0)>0$, and hence
$\varlimsup\limits_{r\to0^+}\sup\limits_{Q\in\partial B}h(r,Q)>0$. 
On the other hand, the perturbation regions are separated and 
their sizes form a geometric sequence, so their total contribution 
to the Carleson norm is summable. This gives the global 
boundedness of the Carleson quantity while preventing it 
from vanishing at small scales.

The remainder of this article is organized as follows.

Section \ref{2} is devoted to the proof of Theorem \ref{thm:main-counterexample}.
To be precise, in Subsection \ref{2.1} we prove a pullback lemma showing that a bi-Lipschitz map
which fixes both the boundary of $B$ and the origin ${\bf0}$ and transforms the Laplacian into a uniformly elliptic operator in divergence form while preserving the elliptic measure at the origin.
In Subsection \ref{2.2}, we construct a specific boundary-fixing bi-Lipschitz map
used in the proof of Theorem \ref{thm:main-counterexample}. Then, in Subsection \ref{2.3},
we give the proof of Theorem \ref{thm:main-counterexample} by using the pullback lemma
given in Subsection \ref{2.1} and the boundary-fixing bi-Lipschitz map
constructed in Subsection \ref{2.2}.

At the end of this section, we make some notational conventions. Let
${\mathbb N}:=\{1,2,\ldots\}$. We always denote by $C$ a \emph{positive constant}
which is independent of the main parameters involved, but it may vary from line to line.
We also use $C_{\alpha,\beta,...}$ to denote a positive constant depending on
the indicated parameters $\alpha$, $\beta\ldots$. The notation $f\lesssim g$ means that $f\le Cg$.
If $f\lesssim g$ and $g\lesssim f$, we then write $f\sim g$. If $f\le Cg$ and $g=h$
or $g\le h$, we then write $f\lesssim g=h$ or $f\lesssim g\le h$.
For any $x\in{\mathbb{R}^n}$ and $r\in(0,{\infty})$, we denote by
$B(x,r):=\left\{y\in{\mathbb{R}^n}:\ |y-x|<r\right\}$
the ball with center $x$ and radius $r$. For any subset $E$ of $\mathbb{R}^n$,
the \emph{symbol} $\mathbf{1}_E$ denotes its \emph{characteristic function} and
$E^\complement$ its \emph{complement} in $\mathbb{R}^n$. Denote by $\mathbf{0}$
the \emph{origin} of ${\mathbb{R}^n}$. For any $x\in\mathbb{R}^n$ and any nonempty
sets $E_1,E_2\subset\mathbb{R}^n$, let
$$
\mathrm{dist\,}(x,E_1):=\inf\{|x-y|:y\in E_1\},$$
and
$$\mathrm{dist\,}(E_1,E_2):=\inf\{|y-z|: y\in E_1,z\in E_2\}.$$
Moreover, we use $\varlimsup$ to denote the \emph{limit superior}.
Finally, in all proofs we consistently retain the notation
introduced in the original theorem (or related statement).

\section{Proof of Main Results}\label{2}

The main task of this section is to prove Theorem \ref{thm:main-counterexample},
which consists of three subsections. In Subsection \ref{2.1}, we establish
a pullback lemma, which shows that
a bi-Lipschitz map, with fixing both the boundary and the pole, transforms the Laplacian into a
uniformly elliptic operator in divergence form without changing the elliptic measure at the pole.
In Subsection \ref{2.2}, we construct a specific bi-Lipschitz map with fixing both the boundary and the pole.
Finally, in Subsection \ref{2.3}, we show Theorem \ref{thm:main-counterexample} by using the pullback lemma
given in Subsection \ref{2.1} and the bi-Lipschitz map constructed in Subsection \ref{2.2}.

\subsection{A Pullback Lemma for Boundary-Fixing Maps}\label{2.1}

To state the main results of this subsection, we first recall some necessary
notation. Let $U\subset\mathbb{R}^n$ be an open set. Denote by the
\emph{symbol $L^1_{\mathrm{loc}}(U)$} the set
of all locally integrable functions on $U$, by the \emph{symbol $C(U)$}
the space of all continuous real-valued functions
on $U$, by the \emph{symbol $C^{\infty}(U)$} the space of all smooth
real-valued functions on $U$, and by the \emph{symbol
$C_{\rm c}^{\infty}(U)$} the space of all smooth functions with compact support in $U$.
The \emph{Sobolev space $W^{1,2}(U)$} consists of all $L^2(U)$ functions
whose first weak derivatives also belong to $L^2(U)$.
The \emph{space $W^{1,2}_{\mathrm{loc}}(U)$} denotes the local Sobolev space
on $U$. The \emph{space $W^{1,2}_0(U)$} is the \emph{closure} of $C_{\rm c}^{\infty}(U)$ in $W^{1,2}(U)$.

Let $M$ be an $n\times n$ real-valued matrix. Denote by
$$
|M|:=\sup\{|M\xi|:\xi\in\mathbb{R}^n,\ |\xi|=1\}
$$
the \emph{operator norm} of $M$, by $M^T$ its \emph{transpose},
and by $\det M$ its \emph{determinant}.
Furthermore, whenever $M$ is invertible, denote by $M^{-1}$
its \emph{inverse}, and let $M^{-T}:=(M^{-1})^T$.
Moreover, the \emph{singular values} of a real-valued $n\times n$
matrix $M$ are defined as the non-negative square roots
of the eigenvalues of the symmetric matrix $M^TM$.

The aim of this subsection is to construct the desired elliptic
measure in Theorem \ref{thm:main-counterexample}, which is the
following proposition.
\begin{proposition}\label{lem:pullback}
Let $\Phi:=(\Phi_1,\ldots,\Phi_n): \mathbb{R}^n\to\mathbb{R}^n$ be a bi-Lipschitz homeomorphism
with Lipschitz constants $0<\lambda\leq \Lambda<\infty$, i.e.,
for any $x,y\in\mathbb{R}^n,$
\begin{align*}
\lambda|x-y|\leq|\Phi(x)-\Phi(y)|\leq \Lambda|x-y|.
\end{align*}
Assume that $\Phi$ satisfies the following  conditions:
\begin{enumerate}[{\rm(a)}]
\item $\Phi(B)=B$;
\item $\Phi|_{\partial B}=\operatorname{Id}_{\partial B}$,
where $\operatorname{Id}_{\partial B}$
denotes the identity map on $\partial B$ (the boundary of $B$);
\item $\Phi(\mathbf{0})=\mathbf{0}$;
\item $\Phi$ is $C^1$ in  $B$;
\item for any $x\in B$, $\det D\Phi(x)\in (0,\infty).$
Here, and thereafter, $D\Phi(x)$ denotes the $n\times n$ matrix defined by setting
$D\Phi(x):=(\partial_j\Phi_i(x))_{1\leq i,j\leq n}.$
\end{enumerate}
For any  $x\in B$, let  $M_x:=D\Phi(x)$, $J_\Phi(x):=\det D\Phi(x)$,  and
$A_\Phi(x):=J_\Phi(x) M_x^{-1}M_x^{-T}.$
Let $L_\Phi:=\mathrm{div}\,(A_\Phi\nabla \cdot)$ be the divergence form elliptic operator associated with $A_\Phi$.
Then the following statements hold.
\begin{enumerate}[{\rm(i)}]
\item The matrix  $A_{\Phi}$ is real-valued, symmetric,  and uniformly elliptic in $B$.
\item Let $f\in C(\partial B)$
and $H_f\in C^{\infty}(B)\cap C(\overline B)$ be the unique solution of the Dirichlet problem
\begin{align}\label{eq:lap}
\begin{cases}
\Delta H_f=0&\ \text{in}\ \ B,\\
H_f=f&\ \text{on}\ \ \partial B.
\end{cases}
\end{align}
Here $\overline B$ denotes the closure of $B$ in $\mathbb{R}^n$.
Let $u_f:=H_f\circ \Phi.$
Then $u_f\in W^{1,2}_{\mathrm{loc}}(B)\cap C(\overline B)$ and $u_f$ is the unique
weak solution of the Dirichlet problem
\begin{align}\label{yangmei2}
\begin{cases}
L_\Phi u=0&\ \text{in}\ \ B,\\
u=f&\ \text{on}\ \ \partial B.
\end{cases}
\end{align}
\item The elliptic measure of $L_{\Phi}$ in $B$ with pole at ${\bf0}$ is the normalized
surface measure on $\partial B$, i.e.,
$\omega_{L_{\Phi}}^{\mathbf{0}}=\frac{\sigma}{\sigma(\partial B)}.$
Consequently,
\begin{align*}
k_{L_{\Phi}}\equiv \frac{1}{\sigma(\partial B)}
\quad \text{and}\quad
\log k_{L_{\Phi}}\in \operatorname{VMO}(\partial B,\sigma).
\end{align*}
\end{enumerate}
\end{proposition}

To show Proposition \ref{lem:pullback}, we need the following Lemma \ref{lem:lap}, which
is a part of \cite[Theorems 1.17 and 1.21]{ABR2001}.

\begin{lemma}\label{lem:lap}
For any given  $f\in C(\partial B)$, there exists the unique
solution $H_f\in C(\overline B)\cap C^{\infty}(B)$ to the Dirichlet problem \eqref{eq:lap}.
Moreover,
\begin{align*}
H_f({\bf0})=\frac{1}{\sigma(\partial B)}\int_{\partial B}f\,d\sigma.
\end{align*}
Consequently,
\begin{align*}
\omega_{\Delta}^{\mathbf{0}}=\frac{\sigma}{\sigma(\partial B)}
\ \ \text{and}\ \
k_{\Delta}\equiv \frac{1}{\sigma(\partial B)}.
\end{align*}
\end{lemma}

Moreover, from \cite[Proposition 9.6 and Lemma 9.5]{b11}, we deduce the following two conclusions.

\begin{lemma}\label{lem:sobolev-pullback}
Let $U,V\subset\mathbb{R}^n$ be open and let $\Psi: U\to V$ be a $C^1$ diffeomorphism.
Let  $u\in W^{1,2}_{\mathrm{loc}}(V)$ and
$w:=u\circ \Psi.$
Then $w\in W^{1,2}_{\mathrm{loc}}(U)$.
Moreover, for almost every $x\in U$,
\begin{align*}
\nabla w(x)=D\Psi(x)^T\nabla u(\Psi(x)).
\end{align*}
\end{lemma}
\begin{lemma}\label{lem:lipschitz-compact-support}
Let $U\subset\mathbb{R}^n$ be open, and let $\psi\in C^1_{\mathrm{c}}(U)$.
Then $\psi\in W^{1,2}_0(U)$.
\end{lemma}

As a part of \cite[Section 27]{d69}, we have the following lemma.

\begin{lemma}\label{lem:weyl}
Let $\Omega\subset\mathbb{R}^n$ be an open set. Assume that $u\in L^1_{\mathrm{loc}}(\Omega)$ satisfies,
for any $\varphi\in C^\infty_{\mathrm{c}}(\Omega)$,
\begin{align*}
\int_{\Omega} u(x)\Delta\varphi(x)\,dx=0.
\end{align*}
Then $u\in C^{\infty}(\Omega)$ and $\Delta u=0$ in $\Omega$.
\end{lemma}
We show the following lemma.
\begin{lemma}\label{lem:matrix-bounds}
Let $\Phi: \mathbb{R}^n\to\mathbb{R}^n$ be bi-Lipschitz with Lipschitz constants $0<\lambda\le \Lambda<\infty$.
Assume that $\Phi$ is differentiable at a point $x\in\mathbb{R}^n$.
Let $M_x:=D\Phi(x).$ Then the following assertions hold.
\begin{enumerate}[{\rm(i)}]
\item For any $\xi\in\mathbb{R}^n$,
$\lambda|\xi|\le |M_x\xi|\le \Lambda|\xi|.$
\item The matrix $M_x$ is invertible.
\item Every singular value of $M_x$ belongs to the interval $[\lambda,\Lambda]$.
\item For any $\xi\in\mathbb{R}^n$,
\begin{align*}
\Lambda^{-1}|\xi|\le\left|M_x^{-1}\xi\right|\le \lambda^{-1}|\xi|
\ \ \text{and}\ \
\Lambda^{-1}|\xi|\le \left|M_x^{-T}\xi\right|\le \lambda^{-1}|\xi|.
\end{align*}
\item If $\det M_x\in (0,\infty)$, then $\det M_x\in [\lambda^n,\Lambda^n].$
\end{enumerate}
\end{lemma}

\begin{proof}
Let $\xi\in\mathbb{R}^n$. Since $\Phi$ is differentiable at $x$,  it  follows that
\begin{align}\label{e2.1}
\lim_{t\to0}\frac{|\Phi(x+t\xi)-\Phi(x)-tM_x\xi|}{|t|}=0.
\end{align}
On the other hand, by the assumption that $\Phi$ is bi-Lipschitz, we find that,
for any $x,\xi\in\mathbb{R}^n$ and $t\in\mathbb{R},$
\begin{align}\label{e2.2}
\lambda|t||\xi|\le|\Phi(x+t\xi)-\Phi(x)|\le \Lambda|t||\xi|.
\end{align}
Combining \eqref{e2.1} and \eqref{e2.2} and letting $t\to0$ in \eqref{e2.2},
we conclude that (i) holds.

Meanwhile, from (i) we infer that  $M_x$ is injective, which implies that $M_x$
is invertible. Thus, (ii) holds.

Let $\kappa$ be an eigenvalue of the symmetric matrix $M_x^TM_x$ and let $\eta\ne 0$
be an eigenvector such that $M_x^TM_x\eta=\kappa\eta.$ By (i), we have
\begin{align*}
\lambda^2|\eta|^2\leq\left|M_x\eta\right|^2=\eta\cdot M_x^TM_x\eta
=\kappa|\eta|^2\leq \Lambda^2|\eta|^2.
\end{align*}
This implies that $\kappa\in [\lambda^2,\Lambda^2],$  which shows (iii).

Furthermore, from (i) and (ii) we immediately deduce that, for any $\xi\in\mathbb{R}^n,$
$\Lambda^{-1}|\xi|\le |M_x^{-1}\xi|\le \lambda^{-1}|\xi|.$
Using this, we  find that, for any $\xi\in\mathbb{R}^n,$
\begin{align}\label{e2.3}
\left|M_x^{-T}\xi\right|=\sup_{|\eta|=1}\left|M_x^{-T}\xi\cdot \eta\right|
=\sup_{|\eta|=1}\left|\xi\cdot M_x^{-1}\eta\right|\leq
|\xi|\sup_{|\eta|=1}\left|M_x^{-1}\eta\right|\leq \lambda^{-1}|\xi|.
\end{align}
Moreover, by (i), we conclude that, for any $\xi\in\mathbb{R}^n$,
\begin{align*}
\left|M_x^T\xi\right|=\sup_{|\eta|=1}\left|M_x^T\xi\cdot \eta\right|
=\sup_{|\eta|=1}\left|\xi\cdot M_x\eta\right|\le
|\xi|\sup_{|\eta|=1}\left|M_x\eta|\le \Lambda|\xi\right|.
\end{align*}
This implies that
\begin{align*}
|\xi|=\left|M_x^TM_x^{-T}\xi\right|\le \Lambda\left|M_x^{-T}\xi\right|,
\end{align*}
which, combined with \eqref{e2.3}, further implies that (iv) holds.

Let $\{s_i\}_{i=1}^n$ denote the singular values of $M_x$. Then $\{s_i\}_{i=1}^n$ are
eigenvalues of $M_x^TM_x$. From this, it follows that
\begin{align}\label{e2.4}
\det\left(M_x^TM_x\right)=\prod_{j=1}^n s_j^2.
\end{align}
On the other hand, note that
\begin{align*}
\det\left(M_x^TM_x\right)=\det\left(M_x^T\right)\det\left(M_x\right)=\left(\det M_x\right)^2.
\end{align*}
This, together with \eqref{e2.4}, further implies that
\begin{align*}
\left(\det M_x\right)^2=\det\left(M_x^TM_x\right)=\prod_{j=1}^n s_j^2.
\end{align*}
By this, the assumption that $\det M_x\in (0,\infty)$, and (iii), we conclude the
desired conclusion (v) holds. This hence finishes the proof of Lemma \ref{lem:matrix-bounds}.
\end{proof}

Now, we prove  Proposition \ref{lem:pullback} by using Lemmas \ref{lem:lap}, \ref{lem:sobolev-pullback},
\ref{lem:lipschitz-compact-support}, \ref{lem:weyl}, and \ref{lem:matrix-bounds}.

\begin{proof}[Proof of Proposition \ref{lem:pullback}]
Obviously, $A_\Phi$ is real-valued and  symmetric.
Note that,  for any $x,\xi\in\mathbb{R}^n,$
\begin{align*}
A_\Phi(x)\xi\cdot\xi=J_\Phi(x)M_x^{-1}M_x^{-T}\xi\cdot\xi
=J_\Phi(x)\left|M_x^{-T}\xi\right|^2.
\end{align*}
Using this and both (iv) and (v) of Lemma \ref{lem:matrix-bounds}, we conclude that,
for any $x,\xi\in\mathbb{R}^n,$
\begin{align*}
\lambda^n\Lambda^{-2}|\xi|^2\leq A_\Phi(x)\xi\cdot \xi=J_\Phi(x)\left|M_x^{-T}\xi\right|^2
\le \Lambda^n\lambda^{-2}|\xi|^2.
\end{align*}
Thus,  $A_\Phi$ is  uniformly elliptic in $B$. This proves (i).

Next, we prove (ii). It is easy to verify that $u_f\in C(\overline B)$ and
$u_f|_{\partial B}=f$. Since $\Phi$ is  $C^1$ on $B$ and $H_f\in C^\infty(B)$, it follows
that $u_f\in C^1(B)\subset W^{1,2}_{\mathrm{loc}}(B)$ and, for any $x\in B,$
\begin{align*}
\nabla u_f(x)=D\Phi(x)^T \nabla H_f(\Phi(x))=M_x^T\nabla H_f(\Phi(x)),
\end{align*}
which implies that, for any $\varphi\in C^\infty_{\mathrm{c}}(B)$ and $x\in B$,
\begin{align}\label{yangmei}
A_\Phi(x)\nabla u_f(x)\cdot \nabla \varphi(x)
&=J_\Phi(x) M_x^{-1}M_x^{-T}M_x^T\nabla H_f(\Phi(x))\cdot \nabla \varphi(x)\notag\\
&=J_\Phi(x)\nabla H_f(\Phi(x))\cdot M_x^{-T}\nabla \varphi(x).
\end{align}
For any $\varphi\in C_{\rm c}^{\infty}(B)$, let  $\psi:=\varphi\circ \Phi^{-1}.$
By Lemma \ref{lem:lipschitz-compact-support}, we find that  $\psi\in C^1_{\mathrm{c}}(B)
\subset W^{1,2}_{0}(B)$ and, for any $x\in B,$
\begin{align*}
\nabla \varphi(x)=D\Phi(x)^{T}\nabla \psi(\Phi(x))
=M_x^T\nabla \psi(\Phi(x)).
\end{align*}
From this, \eqref{yangmei}, and the assumption that $H_f$ is the solution of the Dirichlet
problem \eqref{eq:lap}, we infer that, for any $\varphi\in C_{\rm c}^{\infty}(B),$
\begin{align}\label{wulai}
\int_{B}A_\Phi(x)\nabla u_f(x)\cdot \nabla \varphi(x)\,dx
&=\int_{B}\det D\Phi(x)\nabla H_f(\Phi(x))\cdot \nabla\psi(\Phi(x))\,dx\notag\\
&=\int_{B}\nabla H_f(x)\cdot \nabla\psi(x)\,dx=0.
\end{align}
Thus, $u_f$ is a weak solution of the Dirichlet problem \eqref{yangmei2}.

It remains to show the uniqueness. Let
$w\in C(\overline B)\cap W^{1,2}_{\mathrm{loc}}(B)$ be another weak solution
to \eqref{yangmei2}. Let $h:=w\circ \Phi^{-1}.$ We claim that $h$ is a
solution to the Dirichlet problem \eqref{eq:lap}. If this claim holds, then, by the uniqueness
part of Lemma \ref{lem:lap}, we conclude that $h=H_f$ and
$$w=h\circ \Phi=H_f\circ \Phi=u_f$$
on $B$, which proves the uniqueness. Now, we prove the above claim.
It is easy to verify that $h\in C(\overline B)$ and $h|_{\partial B}=f.$
By the inverse function theorem (see, for instance, \cite[Theorem 9.24]{r76})
and the facts that $\Phi$ is $C^1$ in $B$ and $\det D\Phi(x)\in (0,\infty)$ for any $x\in B$,
we  conclude that $\Phi^{-1}:B\to B$ is a $C^1$ diffeomorphism.
From this and Lemma \ref{lem:sobolev-pullback}, we deduce that $h\in  W^{1,2}_{\mathrm{loc}}(B)$
and, for almost every $x\in B,$
\begin{align*}
\nabla w(x)=D\Phi(x)^T\nabla h(\Phi(x))=M_x^T\nabla h(\Phi(x)).
\end{align*}
For any  $\theta\in C_{\rm c}^\infty(B)$, let $\varphi:=\theta\circ \Phi.$
By Lemma \ref{lem:lipschitz-compact-support}, we find that $\varphi\in C_{\rm c}^1(B)
\subset W_0^{1,2}(B)$. Then an argument similar to that used 
in the proof of \eqref{wulai} yields,
for any $\theta\in C_{\rm c}^\infty(B),$
\begin{align*}
\int_B \nabla h(y)\cdot \nabla \theta(y)\,dy
&=\int_BJ_\Phi(x)\nabla h(\Phi(x))\cdot\nabla \theta(\Phi(x))\,dx\\
&=\int_B J_\Phi(x)M_x^{-1}M_x^{-T}M_x^T\nabla h(\Phi(x))\cdot
M_x^T\nabla \theta(\Phi(x))\,dx\\
&=\int_B A_\Phi(x)\nabla w(x)\cdot \nabla \varphi(x)\,dx=0.
\end{align*}
From this and Lemma \ref{lem:weyl}, we infer that $h\in C^\infty(B)$ 
and $\Delta h=0$ in $B$. Thus, $h$ is a solution to the Dirichlet 
problem \eqref{eq:lap}, which completes the proof of (ii).

Finally, we show (iii). By Lemma \ref{lem:lap}, we conclude that
\begin{align*}
u_f({\bf0})=H_f({\bf0})=\frac{1}{\sigma(\partial B)}\int_{\partial B}f\,d\sigma.
\end{align*}
From this, the definition of elliptic measures, and the uniqueness of the
Riesz--Markov representation
theorem (see, for instance, \cite[Theorem 7.17]{folland1995}), we deduce that
\begin{align*}
\omega_{L_{\Phi}}^{\bf0}=\frac{\sigma}{\sigma(\partial B)}.
\end{align*}
This proves (iii), and hence finishes the proof of Proposition \ref{lem:pullback}.
\end{proof}

\subsection{Construction of the Boundary-Fixing Bi-Lipschitz Map}\label{2.2}

In this subsection, we construct the function $\Phi$ in Proposition \ref{lem:pullback}.
Fix the boundary point
\begin{align*}
Q_0:=e_n:=(0,\ldots,0,1)\in \partial B.
\end{align*}
For any $k\in\mathbb{N}$, let
\begin{align*}
\rho_k:=100^{-k}\quad \text{and}\quad c_k:=(1-4\rho_k)e_n\in B.
\end{align*}
For any $k\in\mathbb{N}$, define
\begin{align}\label{e2.4x}
U_k:=B\left(c_k,\frac{\rho_k}{10}\right)\ \ \text{and}\ \
E_k:=B\left(c_k,\frac{\rho_k}{40}\right).
\end{align}
Choose $\chi\in C_{\rm c}^{\infty}(B(\mathbf{0},1/10))$ such that $\chi\equiv 1$
on $B(\mathbf{0},1/40).$ For any $x:=(x_1,\ldots,x_n)\in\mathbb{R}^n$, let
\begin{align*}
\eta(x):=x_1\chi(x).
\end{align*}
Let $\varepsilon\in (0,\infty)$. For any $k\in\mathbb{N}$ and $x\in\mathbb{R}^n,$ let
\begin{align}\label{shiyan}
\psi_k(x):=\varepsilon\rho_k\eta\left(\frac{x-c_k}{\rho_k}\right),\ \
\Psi(x):=\sum_{k\in\mathbb{N}}\psi_k(x),
\ \ \text{and}\ \ \Phi(x):=x+\Psi(x)e_1.
\end{align}
It is easy to verify that, for any $k\in\mathbb{N},$ $\mathrm{supp}\,(\psi_k)\subset U_k$
and $\{U_k\}_{k\in\mathbb{N}}$ are pairwise disjoint. Thus, for any given $x\in\mathbb{R}^n,$
at most one item in the definition of $\Psi$ is non-zero.

\begin{proposition}\label{lem:Phi}
Let  $\Phi$ be the same as in \eqref{shiyan} and
$\varepsilon\in (0,\frac{1}{4\|\, |\nabla \eta|\,\|_{L^{\infty}(\mathbb{R}^n)}})$.
Then the following statements hold.
\begin{enumerate}[{\rm(i)}]
\item $\Phi: \mathbb{R}^n\to\mathbb{R}^n$ is a bi-Lipschitz homeomorphism with Lipschitz constants
$$0<1-2\|\, |\nabla \eta|\,\|_{L^{\infty}(\mathbb{R}^n)} \varepsilon<
1+2\|\, |\nabla \eta|\,\|_{L^{\infty}(\mathbb{R}^n)} \varepsilon<\infty.$$
\item For any $x\in\mathbb{R}^n\setminus (\bigcup_{k\in\mathbb{N}}U_k)$, $\Phi(x)=x$.
In particular, for any $x\in\mathbb{R}^n\setminus B,$ $\Phi(x)=x$.
\item $\Phi(\mathbf{0})=\mathbf{0}$ and $\Phi(B)=B$.
\item $\Phi$ is  $C^{\infty}$ in  $B$.
\item For any $x\in B$, $\det D\Phi(x)\in (0,\infty)$.
\item For any $k\in\mathbb{N}$ and $x\in E_k,$
\begin{align*}
D\Phi(x)=\operatorname{diag}(1+\varepsilon,1,\ldots,1).
\end{align*}
\end{enumerate}
\end{proposition}

\begin{proof}
We first show (i). Let $L:=\|\, |\nabla \eta|\,\|_{L^{\infty}(\mathbb{R}^n)}$.
Observe that, for any $k\in\mathbb{N}$ and $x\in\mathbb{R}^n,$
$$\nabla \psi_k(x)=\varepsilon\nabla \eta\left(\frac{x-c_k}{\rho_k}\right).$$
This further implies that, for any $k\in\mathbb{N},$ $\psi_k$ is a Lipschitz function
on $\mathbb{R}^n$ with Lipschitz constant at most $\varepsilon L$. Using this
and the definition of $\Psi$, we conclude that $\Psi$ is also a Lipschitz function
on $\mathbb{R}^n$ with Lipschitz constant at most $2\varepsilon L.$

Let $y\in\mathbb{R}^n.$ For any $x\in\mathbb{R}^n,$  let
$$T_y(x):=y-\Psi(x)e_1.$$
Then the Lipschitz constant of $T_y$ is at most $2L\varepsilon\in (0,1)$.
Thus, the map $T_y:\ \mathbb{R}^n\to\mathbb{R}^n$ is compressed.
By the Banach fixed-point theorem (see, for instance, \cite[Theorem 9.23]{r76}),
we find that there exists a unique $x\in\mathbb{R}^n$ such that $T_y(x)=x,$ which implies that
$$\Phi(x)=x+\Psi(x)e_1=y.$$
Thus, $\Phi$ is surjective. Moreover, since
$\varepsilon\in (0,\frac{1}{4L})$, it follows that, for any $x,y\in\mathbb{R}^n,$
\begin{align*}
|\Phi(x)-\Phi(y)|\ge |x-y|-|\Psi(x)-\Psi(y)|\ge (1-2L\varepsilon)|x-y|
\end{align*}
and
\begin{align*}
|\Phi(x)-\Phi(y)|\le |x-y|+|\Psi(x)-\Psi(y)|\le (1+2L\varepsilon)|x-y|.
\end{align*}
Thus, $\Phi$ is injective and bi-Lipschitz. This proves (i).

Moreover, from the assumption that, for any $k\in\mathbb{N}$, 
$\mathrm{supp}\,(\psi_k)\subset U_k$,
we infer that (ii) holds.

Furthermore, by (ii) and the assumption that $\mathbf{0}\notin \bigcup_{k\in\mathbb{N}} U_k$,
we have $\Phi(\mathbf{0})=\mathbf{0}$. Let $x\in B.$ Assume that $\Phi(x)\not\in B$.
Then, from (ii), we deduce that  $\Phi(\Phi(x))=\Phi(x).$ On the other hand, by (i),
we conclude that $\Phi$ is bijective on $\mathbb{R}^n$. Thus, $\Phi(x)=x\in B.$
This contradicts  the assumption  $\Phi(x)\not\in B.$
Therefore, for any $x\in B$, $\Phi(x)\in B$. This shows $\Phi(B)\subset B.$
Using this and the facts that $\Phi$ is a bijection of $\mathbb{R}^n$ and is identity
on $\mathbb{R}^n\setminus B$, we find that $\Phi(B)=B.$ This proves (iii).

Next, we show (iv). Let $K\subset B$ be a compact set. Then $\mathrm{dist}\, (K,\partial B)\in (0,\infty).$
Note that, $\lim_{k\to\infty}\rho_k=0$ and $\lim_{k\to\infty}c_k=e_n$. Then there exists
$N\in\mathbb{N}$ such that, for any $k\in\mathbb{N}\cap [N,\infty)$, $K\cap U_k=\emptyset.$
From this and the assumption that $\psi_k\in C^\infty_{\mathrm{c}}(U_k)$ 
for each $k\in\mathbb{N}$,
it follows that $\Psi\in C^\infty(K).$ By the arbitrariness of $K$, we conclude that
$\Psi\in C^\infty(B).$ Thus, $\Phi\in C^\infty(B).$ This proves (iv).

Let $x:=(x_1,\ldots,x_n)\in B.$ Then
$$\Phi(x)=(x_1+\Psi(x),x_2,\ldots,x_n)$$
and $\det D\Phi(x)=1+\partial_1 \Psi(x).$ Meanwhile, we have
$\|\partial_1 \Psi\|_{L^\infty(B)}<\varepsilon L<\frac{1}{4}.$
Thus, $\det D\Phi(x)\in (0,\infty),$ which implies that (v) holds.

Finally, we prove (vi). From the definitions of $\eta$ and $\Psi$, we infer that,
for any $k\in\mathbb{N}$ and $x\in E_k$,
\begin{align*}
\Psi(x)=\psi_k(x)=\varepsilon x_1
\end{align*}
and
\begin{align*}
D\Phi(x)=\mathrm{diag}\,(1+\partial_1 \Psi(x),1,\ldots,1)
=\mathrm{diag}\,(1+\varepsilon,1,\ldots,1).
\end{align*}
This shows (vi), and hence finishes the proof of Proposition \ref{lem:Phi}.
\end{proof}

\subsection{Proof of Theorem \ref{thm:main-counterexample}}\label{2.3}
In this subsection, we give the proof of Theorem \ref{thm:main-counterexample} by
using Propositions \ref{lem:pullback} and \ref{lem:Phi}.

\begin{proof}[Proof of Theorem \ref{thm:main-counterexample}]
Let $\Phi$ be the same as in Lemma \ref{lem:Phi} and
\begin{align*}
\varepsilon\in\left(0,\min\left\{\frac{1}{4\|\,|\nabla\eta|\,\|_{L^\infty(\mathbb{R}^n)}},1\right\}\right].
\end{align*}
For any $x\in B$, let $M_x:=D\Phi(x)$, $J_\Phi:=\det D\Phi(x),$ and
\begin{align*}
A(x):=A_\Phi(x):=J_\Phi M_x^{-1}M_x^{-T}=
\det D\Phi(x)\,D\Phi(x)^{-1}D\Phi(x)^{-T}.
\end{align*}
Denote the divergence form elliptic operator associated with $A$ by $L_1:=\operatorname{div}(A\nabla \cdot).$
Then (i) follows directly from Proposition \ref{lem:pullback}(iii).

Next, we prove (ii). We first show that there exists a positive constant $c_n$,
depending only on $n$, such that, for any $x\in B$,
\begin{align}\label{huaizhong}
|A(x)-I|\leq c_n\|\,|\nabla\eta|\,\|_{L^\infty(\mathbb{R}^n)}\varepsilon.
\end{align}
By the definition of $\Phi$, we find that
\begin{align*}
\left|M_x-I\right|=|D\Phi(x)-I|=|\nabla\Psi(x)|
\leq\|\,|\nabla\eta|\,\|_{L^\infty(\mathbb{R}^n)}\varepsilon.
\end{align*}
Moreover, from (ii) and (iv) of Lemma \ref{lem:matrix-bounds}, we deduce that
$M_x$ is invertible and
\begin{align}\label{ying2}
\max\left\{\left|M_x^{-1}\right|, \left|M_x^{-T}\right|\right\}\leq\left(1-2\|\,|\nabla\eta|\,
\|_{L^\infty(\mathbb{R}^n)}\varepsilon\right)^{-1}\leq 2,
\end{align}
which further implies that
\begin{align}\label{e2.5}
\left|M_x^{-1}-I\right|=\left|M_x^{-1}(M_x-I)\right|\leq\left|M_x^{-1}\right|\,|M_x-I|
\leq2\|\,|\nabla\eta|\,\|_{L^\infty(\mathbb{R}^n)}\varepsilon
\end{align}
and
\begin{align}\label{e2.6}
\left|M_x^{-T}-I\right|=\left|(M_x^{-1}-I)^T\right|
=\left|M_x^{-1}-I\right|\leq2\|\,|\nabla\eta|\,\|_{L^\infty(\mathbb{R}^n)}\varepsilon.
\end{align}
Then, by \eqref{ying2}, \eqref{e2.5}, and \eqref{e2.6}, we conclude that
\begin{align}\label{ying1}
\left|M_x^{-1}M_x^{-T}-I\right|\leq\left|M_x^{-1}-I\right|\,\left|M_x^{-T}\right|
+\left|M_x^{-T}-I\right|\leq6\|\,|\nabla\eta|\,\|_{L^\infty(\mathbb{R}^n)}\varepsilon.
\end{align}
Furthermore, from Lemma \ref{lem:matrix-bounds}(v) and the fact that
$\det M_x\in(0,\infty)$ (see Proposition \ref{lem:Phi}(v)), we infer that
\begin{align*}
\left|J_\Phi(x)-1\right|&\leq
\max\left\{1-\left(1-2\|\,|\nabla\eta|\,\|_{L^\infty(\mathbb{R}^n)}\varepsilon\right)^n,
\left(1+2\|\,|\nabla\eta|\,\|_{L^\infty(\mathbb{R}^n)}\varepsilon\right)^n+1\right\}\\
&\leq
\left(1+2\|\,|\nabla\eta|\,\|_{L^\infty(\mathbb{R}^n)}\varepsilon\right)^n-\left(1-2\|\,
|\nabla\eta|\,\|_{L^\infty(\mathbb{R}^n)}\varepsilon\right)^n\\
&\leq4 n\left(\frac{3}{2}\right)^{n-1}\|\,|\nabla\eta|\,
\|_{L^\infty(\mathbb{R}^n)}\varepsilon.
\end{align*}
This, combined with \eqref{ying1} and \eqref{ying2}, implies that
\begin{align*}
|A(x)-I|&\leq\left|J_\Phi(x)M_x^{-1}M_x^{-T}-M_x^{-1}M_x^{-T}\right|
+\left|M_x^{-1}M_x^{-T}-I\right|\\
&\leq\left|J_\Phi(x)-1\right|\,\left|M_x^{-1}\right|\,\left|M_x^{-T}\right|
+\left|M_x^{-1}M_x^{-T}-I\right|\\
&\leq\left(\frac{16n3^{n-1}}{2^{n-1}}+6\right)\|\,|\nabla\eta|
\,\|_{L^\infty(\mathbb{R}^n)}\varepsilon.
\end{align*}
This proves \eqref{huaizhong}.

Moreover, recall that, for any $x\in B,$
\begin{align*}
a(x):=\sup_{y\in B(x,\delta(x)/2)} |A(y)-I|.
\end{align*}
By Proposition \ref{lem:Phi}(ii), we find that, for any $x\in B\setminus (\bigcup_{k=1}^{\infty}U_k),$
\begin{align}\label{gj}
D\Phi(x)=I\ \ \text{and}\ \ A(x)=I.
\end{align}
For any $k\in\mathbb{N}$, let
\begin{align*}
V_k:=\left\{x\in B: B\left(x,\delta(x)/2\right)\cap U_k\neq\emptyset\right\}
\ \ \text{and}\ \ V:=\bigcup_{k\in\mathbb{N}}V_k.
\end{align*}
Then, from the definition of $V_k$ and \eqref{gj}, we deduce that, for any
$x\in V^\complement,$ $a(x)=0.$ By this and \eqref{huaizhong}, we further conclude
that, for any $x\in B,$
\begin{align}\label{wuqiu}
a(x)\lesssim\varepsilon\mathbf{1}_{V}\leq\varepsilon
\sum_{k\in\mathbb{N}}\mathbf{1}_{V_k}.
\end{align}

Next, we show that, for any $k\in\mathbb{N},$
\begin{align}\label{zhengzhongji}
V_k\subset \left\{x\in B: \frac{13}{5}\rho_k<\delta(x)< \frac{41}{5}\rho_k
\text{ and }|x-Q_0|\leq \frac{41}{5}\rho_k\right\}.
\end{align}
Let $k\in\mathbb{N}$ and $x\in V_k$.
Then there exists $z\in U_k$ such that $z\in B(x,\frac{\delta(x)}{2})$, which
implies that
\begin{align*}
|\delta(x)-\delta(z)|\le |x-z|<\frac{\delta(x)}{2},
\end{align*}
and hence
\begin{align}\label{snk}
\frac{2}{3}\delta(z)<\delta(x)<2\delta(z).
\end{align}
On the other hand, since $z\in U_k$ and $\delta(c_k)=4\rho_k$, it follows that
\begin{align*}
|\delta(z)-4\rho_k|=|\delta(z)-\delta(c_k)|\leq|z-c_k|<\frac{\rho_k}{10}.
\end{align*}
Thus, we have
$$\frac{39}{10}\rho_k<\delta(z)<\frac{41}{10}\rho_k.$$
Using this and \eqref{snk}, we find that
\begin{align}\label{banxiaoshi}
\frac{13}{5}\rho_k<\frac{2}{3}\delta(z)<\delta(x)<2\delta(z)<\frac{41}{5}\rho_k.
\end{align}
Moreover, note that
\begin{align*}
|x-Q_0|\leq|x-z|+|z-c_k|+|c_k-Q_0|<\frac{\delta(x)}{2}+\frac{\rho_k}{10}+4\rho_k
<\frac{41}{5}\rho_k.
\end{align*}
This, together with \eqref{banxiaoshi}, further implies that \eqref{zhengzhongji}
holds.

Now, we prove \eqref{zhugou}.
Let $r\in (0,1)$ and $Q\in \partial B$. Recall that
\begin{align}\label{tiankong}
[h(r,Q)]^2=\frac{1}{\sigma(\Delta(Q,r))}\int_{T(Q,r)}\frac{[a(x)]^2}{\delta(x)}\,dx.
\end{align}
By \eqref{wuqiu}, we conclude that
\begin{align}\label{lige}
\int_{T(Q,r)}\frac{[a(x)]^2}{\delta(x)}\,dx&\lesssim\varepsilon^2
\sum_{k\in\mathbb{N}}\int_{T(Q,r)\cap V_k}\frac{dx}{\delta(x)}.
\end{align}
Let $k\in\mathbb{N}$ satisfy $T(Q,r)\cap V_k\neq \emptyset.$ Then there exists a
point $x_k\in T(Q,r)\cap V_k$. Since $x_k\in T(Q,r)=B(Q,r)\cap B$ and $Q\in B$,
we infer that
\begin{align*}
\delta(x_k)=\operatorname{dist\,}(x_k,\partial B)\le |x_k-Q|<r.
\end{align*}
This, combined with \eqref{zhengzhongji} and the definition of $\rho_k$, implies that
\begin{align}\label{nimeng}
100^{-k}=\rho_k
<
\frac{5}{13}\delta(x_k)
<\frac{5}{13}r.
\end{align}
Let $k_0$ be the smallest integer  such that  \eqref{nimeng} holds.
Then, from this, \eqref{lige}, and \eqref{zhengzhongji}, we deduce that
\begin{align*}
\int_{T(Q,r)}\frac{[a(x)]^2}{\delta(x)}\,dx&\lesssim\varepsilon^2\sum_{k=k_0}^\infty
\int_{T(Q,r)\cap V_k}\frac{dx}{\delta(x)}\leq\varepsilon^2\sum_{k=k_0}^\infty
\int_{V_k}\frac{dx}{\delta(x)}\\
&\lesssim\varepsilon^2\sum_{k=k_0}^\infty\rho_k^{n-1}\sim\varepsilon^2\rho_{k_0}^{n-1}
\lesssim\varepsilon^2r^{n-1}
\sim \varepsilon^2\sigma(\Delta(Q,r)).
\end{align*}
Using this, \eqref{lige}, and \eqref{tiankong}, we then find that
\begin{align*}
\sup_{r\in(0,1)}\sup_{Q\in\partial B}h(r,Q)\lesssim\varepsilon,
\end{align*}
which proves (ii).

Finally, we show (iii). We first claim that, for any $k\in\mathbb{N},$
\begin{align}\label{aiyi}
h(5\rho_k,Q_0)\gtrsim\varepsilon.
\end{align}
If this claim holds, then
\begin{align*}
\varlimsup_{r\to0^+}\sup_{Q\in\partial B}h(r,Q)\geq\varlimsup_{k\to\infty}h(5\rho_k,Q_0)
\gtrsim\varepsilon,
\end{align*}
which proves \eqref{xiajie}, and hence that (iii) holds. Now, we show
\eqref{aiyi}. Let $k\in\mathbb{N}.$
By the triangle inequality, we conclude that, for any  $x\in E_k=B(c_k,\frac{\rho_k}{40})$,
\begin{align}\label{shinian}
\delta(x)\le \delta(c_k)+|x-c_k|<4\rho_k+\frac{\rho_k}{40}=\frac{161}{40}\rho_k
\end{align}
and
\begin{align*}
|x-Q_0|\le |x-c_k|+|c_k-Q_0|<\frac{\rho_k}{40}+4\rho_k<5\rho_k.
\end{align*}
Therefore,
\begin{align*}
E_k\subset B(Q_0,5\rho_k)\cap B=T(Q_0,5\rho_k).
\end{align*}
From this and \eqref{shinian}, it follows that
\begin{align}\label{shinian2}
[h(5\rho_k,Q_0)]^2&=\frac{1}{\sigma(\Delta(Q_0,5\rho_k))}
\int_{T(Q_0,5\rho_k)}\frac{[a(x)]^2}{\delta(x)}\,dx\notag\\
&\gtrsim\frac{1}{\rho_k^{n}}\int_{E_k}[a(x)]^2\,dx.
\end{align}
Let $x\in E_k.$ By Lemma \ref{lem:Phi}(vi), we have
\begin{align*}
D\Phi(x)=\operatorname{diag}(1+\varepsilon,1,\ldots,1).
\end{align*}
Thus,
\begin{align*}
\det D\Phi(x)=1+\varepsilon
\ \ \text{and}\ \
D\Phi(x)^{-1}=\operatorname{diag}\left((1+\varepsilon)^{-1},1,\ldots,1\right),
\end{align*}
which implies that
\begin{align*}
A(x)=\operatorname{diag}\left((1+\varepsilon)^{-1},1+\varepsilon,\ldots,1+\varepsilon\right).
\end{align*}
Using this and the assumption that $\varepsilon\in (0,1]$, we conclude that
\begin{align*}
a(x)=\sup_{y\in B(x,\delta(x)/2)} |A(y)-I|\geq|A(x)-I|\geq
1-(1+\varepsilon)^{-1}=\frac{\varepsilon}{1+\varepsilon}\ge\frac{\varepsilon}{2}.
\end{align*}
From this and \eqref{shinian2}, we infer that
\begin{align*}
[h(5\rho_k,Q_0)]^2\gtrsim\frac{1}{\rho_k^{n}}\int_{E_k}[a(x)]^2\,dx
\gtrsim\varepsilon^2\frac{|E_k|}{\rho_k^{n}}\sim\varepsilon^2,
\end{align*}
which proves \eqref{aiyi}, and hence finishes the proof of Theorem
\ref{thm:main-counterexample}.
\end{proof}

%
%
%

\bigskip

\noindent
Xiaosheng Lin

\smallskip

\noindent
School of Mathematical Sciences, Jimei University,
Xiamen 361005, The People's Republic of China

\smallskip

\noindent {\it E-mail}: \texttt{xslin@jmu.edu.cn}

\bigskip

\noindent Dachun Yang (Corresponding author), Wen Yuan and Yangyang Zhang

\smallskip

\noindent Laboratory of Mathematics and Complex Systems
(Ministry of Education of China),
School of Mathematical Sciences, Institute for Advanced Study,
Beijing Normal University,
Beijing 100875, The People's Republic of China

\smallskip

\noindent{\it E-mails:} \texttt{dcyang@bnu.edu.cn} (D. Yang)

\noindent\phantom{{\it E-mails:}} \texttt{wenyuan@bnu.edu.cn} (W. Yuan)

\noindent\phantom{{\it E-mails:}} \texttt{yangyzhang@bnu.edu.cn} (Y. Zhang)

\bigskip

\noindent Sibei Yang

\medskip

\noindent School of Mathematics and Statistics, Lanzhou University, Lanzhou 730000, The People's Republic of China

\smallskip

\noindent{\it E-mail:} \texttt{yangsb@lzu.edu.cn}


\begin{thebibliography}{99}

\bibitem{ABR2001}
S. Axler, P. Bourdon and W. Ramey, Harmonic Function Theory, Second edition,
Graduate Texts in Mathematics 137, Springer-Verlag, New York, 2001.

\vspace{-0.3cm}

\bibitem{ahmmt20}J. Azzam, S. Hofmann, J. M. Martell, M. Mourgoglou
and X. Tolsa, Harmonic measure and quantitative connectivity: geometric
characterization of the $L^p$-solvability of the Dirichlet problem,
Invent. Math. 222 (2020), 881--993.

\vspace{-0.3cm}

\bibitem{amt17}J. Azzam, M. Mourgoglou and X. Tolsa, The one-phase problem
for harmonic measure in two-sided NTA domains, Anal. PDE 10 (2017), 559--588.

\vspace{-0.3cm}

\bibitem{b11}
H. Brezis, Functional Analysis, Sobolev Spaces and Partial Differential Equations,
Universitext, Springer, New York, 2011.

\vspace{-0.3cm}

\bibitem{btz23}
S. Bortz, T. Toro and Z. Zhao, Optimal Poisson kernel regularity for elliptic
operators with H\"older continuous coefficients in vanishing chord-arc domains,
J. Funct. Anal. 285 (2023), Paper No. 110025, 64 pp.

\vspace{-0.3cm}

\bibitem{btz23a}
S. Bortz, T. Toro and Z. Zhao, Elliptic measures for Dahlberg--Kenig--Pipher operators:
asymptotically optimal estimates, Math. Ann. 385 (2023), 881--919.

\vspace{-0.3cm}

\bibitem{cfk81}
L. Caffarelli, E. Fabes and C. E. Kenig, Completely singular elliptic-harmonic measures,
Indiana Univ. Math. J. 30 (1981), 917--924.

\vspace{-0.3cm}

\bibitem{d77} B. E. Dahlberg, Estimates of harmonic measure,
Arch. Rational Mech. Anal.
65 (1977), 275--288.

\vspace{-0.3cm}

\bibitem{d86} B. E. Dahlberg, On the absolute continuity of elliptic measure,
Amer. J. Math. 108 (1986), 1119--1138.

\vspace{-0.3cm}

\bibitem{d69} W. F. Donoghue, Distributions and Fourier Transforms, Pure and Applied Mathematics 32,
Academic Press, New York, 1969.

\vspace{-0.3cm}

\bibitem{Escauriaza1996}L. Escauriaza,
The $L^p$ Dirichlet problem for small perturbations of the Laplacian,
Israel J. Math. 94 (1996), 353--366.

\vspace{-0.3cm}

\bibitem{fkp91}
R. Fefferman, C. E. Kenig and J. Pipher, The theory of weights and the Dirichlet problem
for elliptic equations, Ann. of Math. (2) 134 (1991), 65--124.

\vspace{-0.3cm}

\bibitem{folland1995}G. B. Folland, Real Analysis: Modern Techniques and Their Applications, 2nd ed.,
Pure and Applied Mathematics, John Wiley \& Sons, Inc., New York, 1999.

\vspace{-0.3cm}

\bibitem{jk82}D. Jerison and C. E. Kenig, The logarithm of the Poisson kernel of a $C^1$
domain has vanishing mean oscillation, Trans. Amer. Math. Soc. 273 (1982), 781--794.
\vspace{-0.3cm}

\bibitem{jn61} F. John and L. Nirenberg, On functions of bounded mean oscillation,
Comm. Pure Appl. Math. 14 (1961), 415--426.

\vspace{-0.3cm}

\bibitem{Kenig1994}
C. E. Kenig, Harmonic Analysis Techniques for Second Order Elliptic Boundary Value Problems,
CBMS Regional Conf. Ser. in Math. 83,
the American Mathematical Society, Providence, RI, 1994.

\vspace{-0.3cm}

\bibitem{kt03}
C. E. Kenig and T. Toro, Poisson kernel characterization of Reifenberg flat chord arc domains,
Ann. Sci. \'Ecole Norm. Sup. (4) 36 (2003), 323--401.

\vspace{-0.3cm}

\bibitem{kt99}
C. E. Kenig and T. Toro, Free boundary regularity for harmonic measures and Poisson kernels,
Ann. of Math. (2) 150 (1999), 369--454.

\vspace{-0.3cm}

\bibitem{kt97} C. E. Kenig and T. Toro, Harmonic measure on locally flat domains,
Duke Math. J. 87 (1997), 509--551.

\vspace{-0.3cm}

\bibitem{mpt14} E. Milakis, J. Pipher and T. Toro,
Perturbations of Elliptic Operators in Chord Arc Domains,
in: Harmonic Analysis and Partial Differential Equations,
Contemp. Math. 612, pp.~143--161, American Mathematical Society,
Providence, RI, 2014.


\vspace{-0.3cm}

\bibitem{r76} W. Rudin, Principles of Mathematical Analysis, Third edition, Internat. Ser. Pure Appl.
Math., McGraw-Hill Book Co., New York--Auckland--D\"usseldorf, 1976.

\vspace{-0.3cm}

\bibitem{s75}
D. Sarason, Functions of vanishing mean oscillation,
Trans. Amer. Math. Soc. 207 (1975), 391--405.

\vspace{-0.3cm}

\bibitem{tt24} X. Tolsa and T. Toro, The two-phase problem for harmonic measure in
VMO and the chord-arc condition, Trans. Amer. Math. Soc. Ser. B 11 (2024), 1294--1315.

\vspace{-0.3cm}

\bibitem{t2010}
T. Toro, Potential analysis meets geometric measure theory,
in: Proceedings of the International Congress of Mathematicians, Volume III,
pp. 1485--1497,								
Hindustan Book Agency, New Delhi, 2010.


\end{thebibliography}
\end{document}